\documentclass[12pt]{amsart}
\usepackage[all]{xy}
\usepackage{amssymb}
\usepackage{amsthm}
\usepackage{hyperref}
\hypersetup{colorlinks=true,linkcolor=blue,citecolor=magenta}
\usepackage{amsmath}
\usepackage{amscd,enumitem}
\usepackage{verbatim}
\usepackage{eurosym}
\usepackage{float}
\usepackage{color}
\usepackage{dcolumn}
\usepackage[mathscr]{eucal}
\usepackage[all]{xy}
\usepackage{hyperref}
\usepackage{bbm}
\usepackage[textheight=8.3in, textwidth=6.5in]{geometry}
\newtheorem*{thm*}{Theorem}
\newtheorem*{conj*}{Conjecture}
\newtheorem{theorem}{Theorem}

\newtheorem*{remark}{Remark}

\newtheorem*{example}{Example}
\newtheorem{corollary}[theorem]{Corollary}
\newcommand{\ord}{\mathrm{ord}}

\newcommand{\Q}{\mathbb{Q}}

\newcommand{\N}{\mathbb{N}}

\newcommand{\s}{\operatorname{sc}_7}

\begin{document}
\title{Class numbers and self-conjugate 7-cores}
\author{Ken Ono and Wissam Raji}
\address{Department of Mathematics, University of Virginia, Charlottesville, VA 22904}
\email{ken.ono691@virginia.edu}
\address{Department of Mathematics, American University of Beirut, Beirut, Lebanon}
\email{wr07@aub.edu.lb}

\thanks{The first author thanks the support of the Thomas Jefferson Fund and the National Science Foundation
(DMS-1601306 ).The second author thanks the support of the Center for Advanced Mathematical Sciences (CAMS) at the American University of Beirut.}

\begin{abstract}
We investigate $\s(n)$, the number of self-conjugate $7$-core partitions of size $n$.
It turns out that $\s(n)=0$ for $n\equiv 7\pmod 8$.
For $n\equiv 1, 3, 5\pmod 8$, with $n\not \equiv 5\pmod 7,$ 
we find that $\s(n)$ is essentially a Hurwitz class number. 
Using recent work of Gao and Qin, we show that
$$
\s(n) =  2^{-\varepsilon(n)-1}\cdot  H(-D_n),
$$         
where $-D_n:=-4^{\varepsilon(n)}(7n+14)$ and $\varepsilon(n):=\frac{1}{2}\cdot(1+(-1)^{\frac{n-1}{2}})$.
This fact implies several corollaries which are of interest. For example, if $-D_n$ is a fundamental discriminant and
$p\not \in \{2, 7\}$ is a prime with $\ord_p(-D_n)\leq 1$, then
 for every positive integer $k$ we have
\begin{equation}\label{star}
\s\left((n+2)p^{2k}-2\right)=\s(n)\cdot \left(1+\frac{p^{k+1}-p}{p-1}-\frac{p^k-1}{p-1}.\left(\frac{-D_n}{p}\right)\right),
\end{equation}
where $\left(\frac{-D_n}{p}\right)$ is the Legendre symbol.
\end{abstract}
\maketitle
\section{Introduction and statement of results}

A {\it partition} of a non-negative integer $n$ is any nonincreasing sequence of positive integers which sum to $n$. The
partition function $p(n)$, which counts the number of partitions of $n$, has been studied extensively in number theory.
In particular, Ramanujan proved that
\begin{displaymath}
\begin{split}
p(5n+4)&\equiv 0\pmod 5,\\
p(7n+5)&\equiv 0\pmod 7,\\
p(11n+6)&\equiv 0\pmod{11}.
\end{split}
\end{displaymath}

Partitions also play a significant role in representation theory (for example, see \cite{KerberJames}). Indeed, partitions of size $n$ are used
to define {\it Young tableaux}, and their combinatorial properties encode the representation theory of the symmetric group $S_n$.
Moreover, the $t$-{\it core partitions} of size $n$ play an important role in number theory (for example, see \cite{G, GranvilleOno, OnoSze}) and the modular representation theory of $S_n$ and $A_n$ (for example, see
Chapter 2 of \cite{KerberJames}, and \cite{FongSrinivasan, GranvilleOno}). Recall that a partition is a $t$-core if none of the hook numbers of
its Ferrers-Young diagram are multiples of $t$. If $p$ is prime, then the existence of a $p$-core of size $n$ is equivalent to the existence
of a defect 0 $p$-block for both $S_n$ and $A_n$.

For positive integers $t$, we let $c_t(n)$ denote the number of $t$-core partitions of size $n$. If $t\in \{2,3\}$, then it is well-known that
$c_t(n)=0$ for {\it almost all} $n\in \N.$ However, if $t\geq 4$, then $c_t(n)>0$ for every positive integer $n$ (see \cite{GranvilleOno}).

The case where $t=4$ is particularly interesting, as these partitions arise naturally in algebraic number theory. As usual, let $H(-D)$ denote the discriminant $-D<0$ {\it Hurwitz class number}. For the discriminants considered here, $H(-D)$ is
the number of inequivalent (not necessarily primitive) positive definite binary quadratic forms with discriminant $-D<0$.
In particular, if $-D\not \in \{-3, -4\}$ is a fundamental discriminant, then $H(-D)$ is the class number of the imaginary quadratic field
$\Q(\sqrt{-D})$. Sze and the first author proved (see Theorem 2 of \cite{OnoSze}) that
\begin{displaymath}
c_4(n)=\frac{1}{2}H(-32n-20).
\end{displaymath}
\noindent Furthermore, for primes $p$ and $N\in \N$ with $\ord_p(N)\leq 1$, they proved (see Corollary 2 of \cite{OnoSze}) for positive integers $k$ that
\begin{equation}\label{mult}
c_4\left(\frac{N p^{2k}-5}{8}\right) = c_4\left(\frac{N-5}{8}\right)\cdot
\left(1+\frac{p^{k+1}-p}{p-1}-\frac{p^k-1}{p-1}\cdot \left(\frac{-N}{p}\right)\right),
\end{equation}
where $\left(\frac{-N}{p}\right)$ is the Legendre symbol. These formulas implied earlier conjectures
of Hirschhorn and Sellers \cite{HirschhornSellers}.

Further relationships between integer partitions and class numbers are expected to be extremely rare.
In this note we find one more instance where $t$-cores and class numbers are intimately related.
To this end, we let
${\text {\rm sc}}_t(n)$ denote number of the {\it self-conjugate} $t$-core partitions of size $n$. These are $t$-core partitions which
are symmetric with respect to the operation which switches the rows and columns of a Ferrers-Young diagram.

To formulate these results, for a positive odd integer $n$ we define the negative discriminant
\begin{equation}\label{Dn}
-D_n:=\begin{cases} -28n-56 \ \ \ \ \ &{\text {\rm if}}\ n\equiv 1\pmod 4,\\
                                 -7n-14 \ \ \ \ \ &{\text {\rm if}}\ n\equiv 3\pmod 4.
                                 \end{cases}
\end{equation}

\begin{theorem}\label{theorem1}
If $n\not \equiv 5\pmod{7}$ is a positive odd integer, then we have
$$
\s(n)=
\begin{cases}
\frac{1}{4}H(-D_n)\ \ \ \  \ \ \ &{\text {\rm if}}\ n\equiv 1  \pmod 4\\
\frac{1}{2}H(-D_n) \ \ \ \ \ &{\text {\rm if}}\ n\equiv 3  \pmod 8\\
0 \ \ \ \ \ \ \ \ \ \ \ \ \ \ \    &{\text {\rm if}}\ n\equiv 7  \pmod 8.
\end{cases}
$$
\end{theorem}
\begin{example} 
Here we consider the case where $n=9$. According to Theorem~\ref{theorem1}, we have that $\s(9)=H(-308)/4.$
One readily finds that there are eight equivalence classes of discriminant $-308$ binary quadratic forms. The reduced forms 
representing these classes are:
\begin{displaymath}
\begin{matrix} 2X^2+2XY+39Y^2,&  3X^2-2XY+26Y^2, & 3X^2+2XY+26Y^2, & 6X^2-2XY+13Y^2,\\
6X^2+2XY+13Y^2,& 7X^2+11Y^2, & 9X^2-4XY+9Y^2,&  9X^2+4XY+9Y^2.
\end{matrix}
\end{displaymath}
There are fourteen $7$-core partitions of $n=9$.
However, only two of them are self-conjugate.
They are (subscripts are the hook numbers):
$$
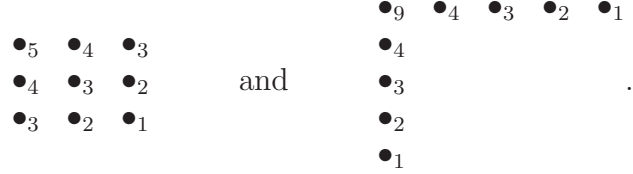
\begin{matrix} \bullet_5 & \bullet_4 & \bullet_3\\
           \bullet_4 & \bullet_3 & \bullet_2\\
           \bullet_3 & \bullet_2 & \bullet_1
           \end{matrix}\ \ \ \ \ \ \ \ {\text {\rm and}}\ \ \ \ \ \ \ \ 
           \begin{matrix} \bullet_9 & \bullet_4 & \bullet_3 & \bullet_2 & \bullet_1\\
           \bullet_4\\ \bullet_3 \\ \bullet_2 \\ \bullet_1\end{matrix}.
$$           
 This example illustrates the conclusion that $\s(9)=2=H(-308)/4$.
\end{example}
\begin{example}
For $n=25$, we have that $-D_n=-756=-3^2\cdot 84$ and $H(-756)=16$. Therefore, Theorem~\ref{theorem1} implies that $\s(25)=H(-756)/4=4.$
\end{example}

Theorem~\ref{theorem1} implies simple short finite formulas for $\s(n)$. To this end, we recall the standard Kronecker character
for a discriminant $D$.
Define the Kronecker character $\chi_D(n)$ for positive integers $n$ by $$\chi_D(n)=\left(\frac{D}{n}\right):=\prod\left(\frac{D}{p_i}\right)^{a_i},$$
where $n=\prod p_i^{a_i}$ and $\left(\frac{D}{p}\right)$ is the Legendre symbol when $p$ is an odd prime and $$\left(\frac{D}{2}\right):=\begin{cases} 0 \ \ \ \ \ \ \ \ \ \ \ \ \ \ \ \ \  \ \ \ \ \  \mbox{if $D$ is even} \\ (-1)^{(D^2-1)/8} \ \ \ \ \ \ \ \ \mbox{if $D$ is odd}.\end{cases}$$

\begin{corollary}\label{FiniteFormula}
If $n\not \equiv 5\pmod 7$ is a non-negative odd integer for which $-D_n$ is a fundamental discriminant, then 
$$
\s(n)=
\begin{cases}
-\frac{1}{4D_n}  \sum_{m=1}^{D_n}\left(\frac{-D_n}{m}\right)m \ \ \ \ \  \ &{\text {\rm if}}\ n\equiv 1 \pmod 4,\\ \ \ \\
-\frac{1}{2 D_n} \sum_{m=1}^{D_n}\left(\frac{-D_n}{m}\right)m, \ \ \ \ \ \   &{\text {\rm if}}\  n\equiv 3  \pmod 8,\\ \ \ \\
0 \ \ \ \ \ \ \ \ \ \ \ \ \  \ \ \ \  \ \ \ \  \ \ \ \ \ \ \ \ \ \ \ \ \ \   &{\text {\rm if}}\ n\equiv 7 \pmod 8.\\
\end{cases}
$$
\end{corollary}

\begin{example}
If $n=11$, then $-D_{11}=-91$. We then find that
$$
\s(11)=-\frac{1}{182}\sum_{m=1}^{91}\left(\frac{-91}{m}\right)m=1.
$$
\end{example}

Hurwitz class numbers enjoy a host of multiplicative properties which can be formulated in terms of the M\"obius function $\mu(d)$ and
the divisor function $\sigma_1(n):=\sum_{1\leq d\mid n} d$.  These imply the following simple corollary which is analogous to (\ref{mult}).

\begin{corollary}\label{MultiplicativeFormula}
If $n\not \equiv 5\pmod 7$ is a positive odd integer, and $-D_n$ is a fundamental discriminant, then for all odd integers $f$ coprime to $7$, we have $$\s((n+2)f^2-2)=\s(n)\sum_{1\leq d\mid f}\mu(d)\left(\frac{-D_n}{d}\right)\sigma_1(f/d).$$
\end{corollary}
\begin{remark}
The cases where $f=p^k$ is a prime power coincide with (\ref{star}).
\end{remark}
\begin{example}
If $n=11$, then we have that $-D_{11}=-91$. Now suppose that $f=15$. Corollary~\ref{MultiplicativeFormula} implies that
$$
\s(2923)=\s(11)\sum_{1\leq d\mid 15} \mu(d)\left(\frac{-91}{d}\right)\sigma_1(15/d).
$$
By direct calculation, since $\s(11)=1$,  the right hand side of this expression equals
$$
\sigma_1(15)+\sigma_1(5)-\sigma_1(3)-\sigma_1(1)=25.
$$
Using the $q$-series identities in the next section (for example (\ref{GenFcn})), one can indeed check directly that $\s(2923)=25$.
\end{example}

\section{Proofs}

Here we prove Theorem~\ref{theorem1} and Corollaries 2 and 3.

\begin{proof}[Proof of Theorem~\ref{theorem1}]
If $t$ is a positive odd integer, then the generating function for $\mathrm{sc}_t(n)$ (see (2) of \cite{Nath}) is
\begin{equation}\label{GenFcn}
\sum_{n=0}^{\infty} \mathrm{sc}_t(n)q^n=\prod_{n=1}^{\infty}\frac{(1-q^{2tn})^{\frac{t-1}{2}}(1+q^{2n-1})}{(1+q^{t(2n-1)})}.
\end{equation}
Therefore, in terms of Dedekind's eta-function $\eta(\tau):=q^{1/24}\prod_{n=1}^{\infty}(1-q^n)$, where $q:=e^{2\pi i \tau}$, we have that
$$S(\tau):=\sum_{n=0}^{\infty}\s(n)q^{n+2}=\frac{\eta(2\tau)^2\eta(14\tau)\eta(7\tau)\eta(28\tau)}{\eta(4\tau)\eta(\tau)}.$$ 
In particular, by the standard theory of modular forms (for example, see Chapter 1.4 of \cite{CBMS}), it follows
that $S(\tau)$ is a holomorphic modular form of weight 3/2 on $\Gamma_0(28)$ with Nebentypus
$\chi_{7}(n):=\left(\frac{7}{\bullet}\right)$.
In terms of theta functions, we have that\footnote{This corrects the statement of Theorem 9 of \cite{Alpoge}.}
$$S(\tau)=\frac{1}{14}\Theta_1(\tau)-\frac{1}{7}\Theta_2(\tau)+\frac{1}{14}\Theta_3(\tau),$$ 
where 
$$\Theta_1(\tau)=\sum_{(x,y,x)\in \mathbb{Z}^3}q^{Q_1(x,y,z)}=\sum_{(x,y,x)\in \mathbb{Z}^3}q^{x^2+y^2+2z^2-yz},$$ $$\Theta_2(\tau)=\sum_{(x,y,x)\in \mathbb{Z}^3}q^{Q_2(x,y,z)}=\sum_{(x,y,x)\in \mathbb{Z}^3}q^{x^2+4y^2+8z^2-4yz}$$ and $$\Theta_3(\tau)=\sum_{(x,y,x)\in \mathbb{Z}^3}q^{Q_3(x,y,z)}=\sum_{(x,y,x)\in \mathbb{Z}^3}q^{2x^2+2y^2+3z^2+2yz+2xz+2xy}.$$
As a result, we get that
\begin{equation}\label{RepNumbers}
\s(n)=\frac{1}{14}R(Q_1; n+2)-\frac{1}{7}R(Q_2; n+2)+\frac{1}{14}R(Q_3; n+2),
\end{equation}
where $R(Q_i; n+2)$ is the number of integral representations of $n+2$ by $Q_i(x,y,z)$.

To prove the theorem, we must relate these three theta functions to the Eisenstein series in
the space $M_{\frac{3}{2}}(28,\chi_7)$. Thankfully, Gao and Qin \cite{GaoQin} have
already carried out these calculations.
They produce three Eisenstein series which form a basis of this space (see Theorem 3.2 of \cite{GaoQin}).
These series are given by their Fourier expansions:
\begin{eqnarray*}
g_1(\tau)&=&1+2\pi \sqrt{7}\sum_{n=1}^{\infty}\lambda(7n,28)\alpha(7n)\left(A(7,7n)-\frac{1}{7}\right)\sqrt{n}q^n,\\ g_2(\tau)&=& \frac{2}{49}\pi\sqrt{7}\sum_{n=1}^{\infty}\lambda(7n,28)\alpha(7n)\sqrt{n}q^n,\\ g_3(\tau)&=&2\pi\sqrt{7}\sum_{n=1}^{\infty}\lambda(7n,28)\left(A(7,7n)-\frac{1}{7}\right)\sqrt{n}q^n.
\end{eqnarray*}
The quantities $\alpha(m), A(p,m), \lambda(7m,28)$, $h_p(m)$ and $h'_p(m)$ are  defined by
$$
\alpha(m):=
\begin{cases}
3\cdot 2^{-\frac{1+h_2(m)}{2}}\ \ \  \ \ &{\text {\rm if}} \ h_2(m) \ {\text {\rm is  odd}},\\
3\cdot 2^{-1-\frac{h_2(m)}{2}} \ \  \ \ &{\text {\rm  if}} \ h_2(m) \ {\text {\rm is  even  and}}\  h_2'(m)\equiv 1  \pmod 4,\\
0,\ \ \  \ \ \ \ \ \ \ \ \ \ \ \ \ \ \ &{\text {\rm if}} \ h_2(m) \ {\text {\rm  is even and}}\  \ h_2'(m)\equiv 3 \pmod 8,\\
2^{-\frac{h_2(m)}{2}}\ \ \  \ \ \ \ \ \ \ &{\text {\rm  if}} \ h_2(m) \ {\text {\rm is  even  and}} \ h_2'(m)\equiv 7 \pmod 8,\\
\end{cases}
$$
\\
\\
$$
A(p,m):=
\begin{cases}
p^{-1}-(1+p)p^{-\frac{3+h_p(m)}{2}} \ \ \ &{\text {\rm if}}\  h_p(m) \ {\text {\rm is  odd,}}\\
p^{-1}-2p^{-1-\frac{h_p(m)}{2}} \ \ \ &{\text {\rm  if}} \ h_p(m) \ {\text {\rm is  even  and}} \left(\frac{-h'_p(m)}{p}\right)=-1,\\
p^{-1}, \ \ \  &{\text {\rm if}} \ h_p(m) \ {\text {\rm  is  even  and}}\  \left(\frac{-h'_p(m)}{p}\right)=1,
\end{cases}
$$
where $h_p(m)$ is the non-negative integer for which $p^{h_p(m)}||m$ and $h'_p(m):=\frac{m}{p^{h_p(m)}}$, and
$$
\lambda(7m,28)=
\begin{cases}
\frac{49}{4\pi\sqrt{7m}}\cdot H(-7m), \ \ \ &{\text {\rm  if}} \ m\equiv 5 \pmod 8,\\
\frac{49}{12\pi\sqrt{7m}}\cdot H(-7m),  \ \ \ &{\text {\rm otherwise.}}
\end{cases}
$$
The formula for  $\alpha(m)$, when $h_2'(m)\equiv 3, 7 \pmod 8,$ corrects a typographical error in \cite{GaoQin}.

Since the three Eisenstein series form a basis of this space, it is trivial to deduce that
$$\Theta_1(\tau)=g_1(\tau)-3g_2(\tau),\  \Theta_2(\tau)=g_1(\tau)-\frac{3}{2}g_3(\tau),\  \Theta_3(\tau)=g_1(\tau)+14g_2(\tau).$$
 Using $h_7(7n)=1$, $h_7'(7n)=n$, $h_2(7n)=0$ and $h_2'(7n)=7n$ for $(n,14)=1$, we find that 
$$A(7,7n)=\frac{1}{7}-\frac{8}{49},$$ and
$$\alpha(7n):=\begin{cases}
\frac{3}{2}\ \ \ \   &{\text {\rm if}} \ n\equiv 3\pmod 4,\\
1 \ \ \ \ &{\text {\rm if}} \ n\equiv 1 \pmod 8,\\
0 \ \ \  &{\text {\rm if}}\  n\equiv 5 \pmod 8.
\end{cases}$$
Combining these facts, for positive integers $n$ we deduce that
$$R(Q_1; n)=2\pi\sqrt{7n}\lambda(7n,28)\left(A(7,7n)-\frac{1}{7}\right)(\alpha(7n)-3),$$ 
$$R(Q_2;n)= 2\pi\sqrt{7n}\lambda(7n,28)\left(A(7,7n)-\frac{1}{7}\right)(\alpha(7n)-\frac{3}{2}),$$ and $$R(Q_3;n)=2\pi\sqrt{7n}\lambda(7n,28)\alpha(7n)\left(A(7,7n)-\frac{1}{7}+\frac{14}{49}\right).$$ 
Finally, these expressions, for positive odd $n\not \equiv 5\pmod 7$, simplify to
$$
R(Q_1; n+2)=
\begin{cases}
2H(-D_n) \ \ \ &{\text {\rm if}}\ n\equiv 1 \pmod 4,\\
8H(-D_n) \ \ \ &{\text {\rm if}}\ n\equiv 3 \pmod 8,\\
4H(-D_n) \ \ \ &{\text {\rm if}}\ n\equiv 7 \pmod 8,
\end{cases}
$$
$$
R(Q_2; n+2)=
\begin{cases}
0 \ \ \ &{\text {\rm if}}\ n\equiv 1 \pmod 4,\\
2H(-D_n) \ \ \ &{\text {\rm if}}\ n\equiv 3 \pmod 8,\\
2H(-D_n) \ \ \ &{\text {\rm if}} \ n\equiv 7\pmod 8,
\end{cases}
$$

$$
R(Q_3; n+2)=
\begin{cases}
\frac{3}{2}H(-D_n) \ \ \ &{\text {\rm if}}\ n\equiv 1 \pmod 4,\\
3H(-D_n) \ \ \ &{\text {\rm if}}\ n\equiv 3\pmod 8,\\
0 \ \ \ &{\text {\rm if}}\ n\equiv 7 \pmod 8.
\end{cases}
$$
The claimed formulas now follow from (\ref{RepNumbers}).
\end{proof}

\begin{proof}[Proof of Corollary~\ref{FiniteFormula}] If $-D<0$ is a fundamental discriminant, then it is
well-known (for example, see (7.29) of \cite{Cox}) that
$$
H(-D)=-\frac{|O_K^{\times}|}{2D}\sum_{m=1}^{D} \left(\frac{-D}{m}\right)m,
$$
where $K=\Q(\sqrt{-D})$ and $O_K^{\times}$ denotes the units in its corresponding ring of integers.
The claim now follows from Theorem~\ref{theorem1}.
\end{proof}

\begin{proof}[Proof of Corollary~\ref{MultiplicativeFormula}] 
If $-D<0$ is a fundamental discriminant of an imaginary quadratic field, then for every integer $f$ it is known
 (p. 273, \cite{Cohen}) that
$$
H(-Df^2)=\frac{H(-D)}{w(-D)}\cdot\sum_{1\leq d\mid f}\mu(d)\left(\frac{-D}{d}\right)\sigma_1(f/d),
$$
where $w(-D)$ denotes half of the number of roots of unity in $\Q(\sqrt{-D})$.
The claim now follows from Theorem~\ref{theorem1}.

\end{proof}


\end{document}